\documentclass[12pt, reqno]{amsart}
\author[P.~Leonetti]{Paolo Leonetti}
\address{
Universit\'a degli Studi dell'Insubria, via Monte Generoso 71, 21100 Varese, Italy}
\email{leonetti.paolo@gmail.com}
\urladdr{\url{https://sites.google.com/site/leonettipaolo/}} 

\makeatletter
\@namedef{subjclassname@2020}{\textup{2020} Mathematics Subject Classification}
\makeatother


\title{Core equality of real sequences}

\usepackage[T1]{fontenc}
\usepackage{amsmath}
\usepackage{amssymb}
\usepackage{amsthm}
\usepackage[left=2.5cm, right=2.5cm, top=3cm]{geometry}
\usepackage{hyperref}
\usepackage{fancyhdr}
\usepackage[inline]{enumitem}
\usepackage{comment}
\usepackage{nicefrac}
\usepackage{bm}
\usepackage{mathrsfs}
\usepackage{graphicx}
\usepackage[utf8]{inputenc}
\usepackage{cancel}
\usepackage{mathtools}

\AtBeginDocument{%
   \def\MR#1{}
}

\newtheorem{thm}{Theorem}[section]
\newtheorem{lem}[thm]{Lemma}

\theoremstyle{definition} 
\newtheorem{defi}[thm]{Definition}
\let\olddefi\defi
\renewcommand{\defi}{\olddefi\normalfont}
\newtheorem{example}[thm]{Example}
\let\oldexample\example
\renewcommand{\example}{\oldexample\normalfont}
\newtheorem{rmk}[thm]{Remark}
\let\oldrmk\rmk
\renewcommand{\rmk}{\oldrmk\normalfont}

\theoremstyle{remark}
\newtheorem{claim}{\textsc{Claim}}
\newtheorem*{claim*}{\textsc{Claim}}

\hypersetup{
    pdftitle={Core equality of real sequences},
    pdfauthor={Paolo Leonetti},
    pdfmenubar=false,
    pdffitwindow=true,
    pdfstartview=FitH,
    colorlinks=true,
    linkcolor=blue,
    citecolor=green,
    urlcolor=cyan
}

\providecommand{\MR}[1]{}

\providecommand{\MR}{\relax\ifhmode\unskip\space\fi MR }

\providecommand{\href}[2]{#2}

\begin{document}

\maketitle
\thispagestyle{empty}

\begin{abstract}
\noindent Given an ideal $\mathcal{I}$ on $\omega$ and a bounded real sequence $\bm{x}$, we denote by $\mathrm{core}_{\bm{x}}(\mathcal{I})$ the smallest interval $[a,b]$ such that $\{n \in \omega: x_n \notin [a-\varepsilon,b+\varepsilon]\} \in \mathcal{I}$ for all $\varepsilon>0$ (which corresponds to the interval $[\,\liminf \bm{x}, \limsup \bm{x}\,]$ if $\mathcal{I}$ is the ideal $\mathrm{Fin}$ of finite subsets of $\omega$). 

First, we characterize all the infinite real matrices $A$ such that 
$$
\mathrm{core}_{A\bm{x}}(\mathcal{J})=\mathrm{core}_{\bm{x}}(\mathcal{I})
$$
for all bounded sequences $\bm{x}$, provided that $\mathcal{J}$ is a countably generated ideal on $\omega$ and $A$ maps bounded sequences into bounded sequences. 
Such characterization fails if both $\mathcal{I}$ and $\mathcal{J}$ are the ideal of asymptotic density zero sets. 
Next, we show that such equality is possible for distinct ideals $\mathcal{I}, \mathcal{J}$, answering an open question in [J.~Math.~Anal.~Appl.~\textbf{321} (2006), 515--523]. Lastly, we prove that, if $\mathcal{J}=\mathrm{Fin}$, the above equality holds for some matrix $A$ if and only if 
$\mathcal{I}=\mathrm{Fin}$ or $\mathcal{I}$ is an isomorphic copy of $\mathrm{Fin}\oplus \mathcal{P}(\omega)$ on $\omega$. 
\end{abstract}

\section{Introduction}

Let $\mathcal{I}$ be an ideal on the natural numbers $\omega$, that is, a family of subsets of $\omega$ closed under subsets and finite unions. Unless otherwise stated, it is assumed that $\mathcal{I}$ contains the family $\mathrm{Fin}$ of finite sets and that $\omega \notin \mathcal{I}$. We denote by $\mathcal{I}^+:=\mathcal{P}(\omega)\setminus \mathcal{I}$ and $\mathcal{I}^\star:=\{S\subseteq \omega: \omega\setminus S \in \mathcal{I}\}$ the family of $\mathcal{I}$-positive sets  and the dual filter of $\mathcal{I}$, respectively. Ideals are regarded as subsets of the Cantor space $\{0,1\}^\omega$, hence it is possible to speak about $F_\sigma$-ideals, analytic ideals, meager ideals, etc. An important example of an ideal is the family $\mathcal{Z}$ of sets $S\subseteq \omega$ with asymptotic density zero, that is, $|S\cap [0,n]|=o(n)$ as $n\to \infty$.

Given a sequence $\bm{x}=(x_n: n \in \omega)$ taking values in a topological space $X$, we define its $\bm{\mathcal{I}}$\textbf{-core} by 
$$
\mathrm{core}_{\bm{x}}(\mathcal{I}):=\bigcap_{S \in \mathcal{I}^\star}\overline{\mathrm{co}}\left\{x_n: n \in S\right\},
$$
where $\mathrm{co}$ denotes the convex hull operator and $\overline{\mathrm{co}}$ its closure, see \cite{MR4126774, MR3955010}. 
In the cases where $\mathcal{I}=\mathrm{Fin}$ and $\mathcal{I}=\mathcal{Z}$, the $\mathcal{I}$-core of $\bm{x}$ is usually called \textquotedblleft Knopp core\textquotedblright\, and \textquotedblleft statistical core,\textquotedblright\, respectively, see \cite{MR1441452, MR1416085, MR0142952, MR0487146}.
Let also $\Gamma_{\bm{x}}(\mathcal{I})$ denote the set of $\mathcal{I}$\emph{-cluster points} of $\bm{x}$, that is, the set of all $\eta \in X$ such that $\{n \in \omega: x_n \in U\} \in \mathcal{I}^+$ for all neighborhoods $U$ of $\eta$. It has been shown in \cite[Corollary 2.3]{MR3955010} that, if $\bm{x}$ is bounded real sequence, then $\mathrm{core}_{\bm{x}}(\mathcal{I})=\mathrm{co}(\Gamma_{\bm{x}}(\mathcal{I}))$. In addition, since $\Gamma_{\bm{x}}(\mathcal{I})$ is a nonempty compact set, see e.g. \cite[Lemma 3.1]{MR3920799}, it follows that 
\begin{equation}\label{eq:Icorerepresentation}
\mathrm{core}_{\bm{x}}(\mathcal{I})=\left[\,\mathcal{I}\text{-}\liminf \bm{x}, \mathcal{I}\text{-}\limsup \bm{x}\,\right]
\end{equation}
for all real bounded sequences $\bm{x}$, where $\mathcal{I}\text{-}\liminf \bm{x}:=\min \Gamma_{\bm{x}}(\mathcal{I})$ and $\mathcal{I}\text{-}\limsup \bm{x}:=\max \Gamma_{\bm{x}}(\mathcal{I})$ (note that, if $\mathcal{I}=\mathrm{Fin}$ then $\mathcal{I}\text{-}\liminf$ coincides with the ordinary $\liminf$, and analogously for $\mathcal{I}\text{-}\limsup$). 

Given an infinite matrix $A=(a_{n,k}: n,k \in \omega)$, we denote by $\mathrm{dom}(A)$ its domain, that is, the set of real sequences $\bm{x}=(x_n)$ for which 
the 
$A$-transformed sequences $A\bm{x}:=\left(A_n\bm{x}: n \in \omega\right)$ are well defined, which means that the series 
$$
A_n\bm{x}:=\sum_{k \in \omega} a_{n,k}x_k 
$$ 
is convergent to a (finite) real number for each $n \in \omega$. We write $A\ge 0$ if $a_{n,k}\ge 0$ for all $n,k \in \omega$. 
Given sequence spaces $\mathcal{A}, \mathcal{B} \subseteq \mathbf{R}^\omega$, we denote by $(\mathcal{A}, \mathcal{B})$ the family of infinite matrices $A$ such that $\mathcal{A}\subseteq \mathrm{dom}(A)$ and $A\bm{x} \in \mathcal{B}$ for all $\bm{x} \in \mathcal{A}$. For instance, it is well known that $A \in (\ell_\infty, \ell_\infty)$ if and only if $A \in (c,c)$ if and only if $\|A\|<\infty$, where $\|A\|:=\sup_n \sum_k|a_{n,k}|$, see e.g. \cite[Theorem 2.3.5]{MR1817226}. Here, as usual, $\ell_\infty$ and $c$ stand for the vector space of bounded real sequences and convergent real sequences, respectively. $\ell_\infty$ and all its subspaces are endowed with the topology induced by the supremum norm. 

We denote by $c_b(\mathcal{I})$ the vector space of bounded real sequences $\bm{x}$ which, in addition, 
are $\mathcal{I}$-convergent to some $\eta \in \mathbf{R}$, namely, $\{n \in \omega: x_n \in U\} \in \mathcal{I}^\star$ for all neighborhoods $U$ of $\eta$ (this will be shortened as $\mathcal{I}\text{-}\lim \bm{x}=\eta$). Note that $c_b(\mathcal{I})=\ell_\infty$ whenever $\mathcal{I}$ is maximal (that is, $\mathcal{I}^\star$ is a free ultrafilter on $\omega$). 
Structural properties of bounded $\mathcal{I}$-convergent sequences, their $\mathcal{I}$-cluster points, and the relationship with $A$-summability have been extensively studied, see e.g. \cite{MR4600193, MR2735533, MR3863065, MR4041540, MR3883309, MR2340330, MR3836186, 
Leo25conv, 
MR4507668, Alwin23} and references therein. 
Given ideals $\mathcal{I}, \mathcal{J}$ on $\omega$, we say that an infinite real matrix $A$ is $\bm{(\mathcal{I}, \mathcal{J})}$\textbf{-regular} if it maps $\mathcal{I}$-convergent bounded sequences into $\mathcal{J}$-convergent bounded sequences preserving the corresponding ideal limits, namely, 
$$
A \in (c_b(\mathcal{I}), c_b(\mathcal{J})) 
\quad \text{ and }\quad 
\mathcal{I}\text{-}\lim \bm{x} = \mathcal{J}\text{-} \lim A\bm{x} \,\,\text{ for all }\bm{x} \in c_b(\mathcal{I}),
$$
see e.g. \cite{ConnorLeo, Leo22, MR4506089}. 
Note that $(\mathrm{Fin}, \mathrm{Fin})$-regular matrices are simply the classical regular matrices. 
Probably the most important regular matrix is the Cesàro matrix $C_1=(a_{n,k})$ defined by $(a_{n,k})=1/n$ if $k \le n$ and $a_{n,k}= 0$ otherwise. 

A characterization of $(\mathcal{I}, \mathcal{J})$-regular matrices has been recently proved by the author and Jeff Connor in \cite[Theorem 1.2 and Theorem 1.3]{ConnorLeo}, see also \cite[Corollary 2.11]{Leo22}. 
\begin{thm}\label{thm:mainsilvermantoeplitz} 
Let $A$ be an infinite real matrix and fix ideals $\mathcal{I}$, $\mathcal{J}$ on $\omega$. Suppose also that $A \ge 0$ or $\mathcal{I}=\mathrm{Fin}$ or 
$\mathcal{J}$ is countably generated.

Then $A$ is $(\mathcal{I}, \mathcal{J})$-regular if and only if:
\begin{enumerate}[label={\rm (\textsc{T}\arabic{*})}]
\item \label{item:ST1} $
\|A\|
< \infty$\textup{;}
\item \label{item:ST2} $\mathcal{J}\text{-}\lim_n \sum_k a_{n,k}=1$\textup{;}
\item \label{item:ST3} $\mathcal{J}\text{-}\lim_n \sum_{k \in E} |a_{n,k}|=0$ for all $E \in \mathcal{I}$\textup{.}
\end{enumerate}
\end{thm}
In the statement above, recall that an ideal $\mathcal{J}$ on $\omega$ is countably generated if there exists a sequence $(Q_k: k \in \omega)$ of subsets of $\omega$ such that $S \in \mathcal{I}$ if and only if $S\subseteq \bigcup_{k \in F}Q_k$ for some finite $F \in \mathrm{Fin}$. Examples of countably generated ideals include $\mathrm{Fin}$ and the isomorphic copies on $\omega$ of $\mathrm{Fin}\times \{\emptyset\}:=\{S\subseteq \omega^2: \exists n\in \omega, S\subseteq [0,n]\times \omega \}$ and  $\mathrm{Fin}\oplus \mathcal{P}(\omega):=\{S\subseteq \{0,1\}\times \omega:|S\cap (\{0\}\times \omega)|<\infty\}$, cf. \cite[Remark 2.16]{Leo22}.  Hereafter, an ideal $\mathcal{I}$ on a countably infinite set $W$ is said to be an isomorphic copy of an ideal $\mathcal{J}$ on $\omega$ if there exists a bijection $f: W \to \omega$ such that $S \in \mathcal{J}$ if and only if $f^{-1}[S] \in \mathcal{I}$ for each $S\subseteq \omega$.)

The above result extends the classical Silverman–Toeplitz characterization, which corresponds to the case $\mathcal{I}=\mathcal{J}=\mathrm{Fin}$. Lastly, it is worth mentioning that Theorem \ref{thm:mainsilvermantoeplitz} does not hold for arbitrarily ideals: indeed, there exists a $(\mathcal{Z}, \mathcal{Z})$-regular which does not satisfy condition \ref{item:ST3}, see \cite[Theorem 1.4]{ConnorLeo}.

\section{Main results}

Given ideals $\mathcal{I}, \mathcal{J}$ on $\omega$, we study the core equality problem
\begin{equation}\label{eq:mainproblemequalitycores}
\mathrm{core}_{A\bm{x}}(\mathcal{J})=\mathrm{core}_{\bm{x}}(\mathcal{I})
\quad \text{ for all sequences }\bm{x} \in \ell_\infty.
\end{equation}
More explicitly, we obtain necessary and sufficient conditions on the entries of $A$ to satisfy equality \eqref{eq:mainproblemequalitycores} and, then, we study the existence of such matrices $A$. 

The first result in this direction has been obtained by Allen \cite{MR0011332}, which provides a characterization of the matrices $A$ which satisfy equality \eqref{eq:mainproblemequalitycores} in the case $\mathcal{I}=\mathcal{J}=\mathrm{Fin}$:
\begin{thm}\label{thm:allen}
Let $A$ be an infinite real matrix and suppose that $\mathcal{I}=\mathcal{J}=\mathrm{Fin}$. 

Then equality \eqref{eq:mainproblemequalitycores} holds if and only if\textup{:}
\begin{enumerate}[label={\rm (\textsc{A}\arabic{*})}]
\item \label{item:Allen1} $A$ is regular\textup{;} 
\item \label{item:Allen2} $\lim_n \sum_k |a_{n,k}|=1$\textup{;}
\item \label{item:Allen3} for each infinite $E\subseteq \omega$, there exists a strictly increasing sequence $(n_i: i \in \omega)$ in $\omega$ such that $\lim_i \sum_{k \in E}a_{n_i,k}=1$
\textup{.}
\end{enumerate}
\end{thm}
Note that, taking into account \ref{item:Allen2}, condition \ref{item:Allen3} can be rewritten equivalently as $\limsup_n \sum_{k \in E}|a_{n,k}|=1$ for all infinite $E\subseteq \omega$. In addition, condition \ref{item:Allen2} implies that $A$ maps bounded sequences into bounded sequences, i.e., $A \in (\ell_\infty, \ell_\infty)$. 

A partial extension of the result above has been obtained by Connor, Fridy, and Orhan in the case where all the entries of $A$ are nonnegative, see \cite[Theorem 2.1]{MR2241135}.
\begin{thm}\label{thm:connorfridyorhan}
Let $A\in (\ell_\infty, \ell_\infty)$ be an infinite real matrix, let $\mathcal{I}$, $\mathcal{J}$ be ideals on $\omega$, and suppose that $A\ge 0$. 

Then equality \eqref{eq:mainproblemequalitycores} holds if and only if\textup{:}
\begin{enumerate}[label={\rm (\textsc{C}\arabic{*})}]
\item \label{item:Connor1} $A$ is $(\mathcal{I}, \mathcal{J})$-regular\textup{;}
\item \label{item:Connor2} $\mathcal{J}\text{-}\limsup_n \sum_{k \in E} a_{n,k}=1$ for all $E \in \mathcal{I}^+$\textup{.}
\end{enumerate}
\end{thm}

Our first main result removes the hypotheses that the entries of $A$ are nonnegative and, on the other hand, it requires that $\mathcal{J}$ is countably generated. Hence it provides a generalization of Theorem \ref{thm:allen}.
\begin{thm}\label{thm:countably generated}
Let $A\in (\ell_\infty, \ell_\infty)$ be an infinite real matrix, let $\mathcal{I}$, $\mathcal{J}$ be ideals on $\omega$, and suppose that $\mathcal{J}$ is countably generated. 

Then equality \eqref{eq:mainproblemequalitycores} holds if and only if\textup{:}
\begin{enumerate}[label={\rm (\textsc{L}\arabic{*})}]
\item \label{item:Leo1} $A$ is $(\mathcal{I}, \mathcal{J})$-regular\textup{;}
\item \label{item:Leo2} $\mathcal{J}\text{-}\limsup_n \sum_{k \in E} |a_{n,k}|=1$ for all $E \in \mathcal{I}^+$\textup{.}
\end{enumerate}
\end{thm}

The proof of Theorem \ref{thm:countably generated} recovers also Theorem \ref{thm:connorfridyorhan}, see Remark \ref{rmk:generalizationifonlyif} below. 
In addition, the above characterization does \emph{not} hold without any constraint on the ideals $\mathcal{I}, \mathcal{J}$. Indeed, it fails for $\mathcal{I}=\mathcal{J}=\mathcal{Z}$, see Remark \ref{rmk:alwaysholds}.

At this point, another result by Connor, Fridy, and Orhan proves that there are \emph{no} regular matrices $A$ satisfying equality \eqref{eq:mainproblemequalitycores} if $\mathcal{I}=\mathcal{Z}$ and $\mathcal{J}=\mathrm{Fin}$, see \cite[Theorem 2.4]{MR2241135}. Accordingly, the authors left as open question whether there exist distinct ideals $\mathcal{I}, \mathcal{J}$ on $\omega$ for which equality \eqref{eq:mainproblemequalitycores} holds for some matrix $A \in (\ell_\infty, \ell_\infty)$. Next, we show that the answer is affirmative. To this aim, recall that an ideal $\mathcal{I}$ is \emph{Rudin--Keisler below} an ideal $\mathcal{J}$, shortened as $\mathcal{I}\le_{\mathrm{RK}} \mathcal{J}$, if there exists a map $h: \omega \to \omega$ such that $S \in \mathcal{I}$ if and only if $h^{-1}[S] \in \mathcal{J}$. 

\begin{thm}\label{thm:RK}
Let $\mathcal{I}$, $\mathcal{J}$ be ideals on $\omega$ such that $\mathcal{I}\le_{\mathrm{RK}} \mathcal{J}$. Then there exists an infinite real matrix $A \in (\ell_\infty, \ell_\infty)$ which satisfies equality \eqref{eq:mainproblemequalitycores}.
\end{thm}

It is worth to recall that, if $\mathcal{J}$ is a $P$-ideal (that is, for all increasing sequences $(J_n)$ in $\mathcal{J}$ there exists $J \in \mathcal{J}$ such that $J_n\setminus J \in \mathrm{Fin}$ for all $n \in \omega$) and $\mathcal{I}\le_{\mathrm{RK}} \mathcal{J}$, then it is possible to choose a witnessing function $h$ which is finite-to-one, see \cite[Proposition 1.3.1]{MR1711328}. Rudin–Keisler ordering and the latter stronger variant (known as Rudin--Blass ordering) on the maximal ideals are extensively studied in the literature, cf. \cite[Section 1.3]{MR1711328} and references therein. 
Examples of (distinct) ideals $\mathcal{I}, \mathcal{J}$ on $\omega$ such that $\mathcal{I}\le_{\mathrm{RK}} \mathcal{J}$ are abundant. For instance, it is known that $\mathrm{Fin} \le_{\mathrm{RK}} \mathcal{J}$ for all meager ideals $\mathcal{J}$, see \cite[Corollary 3.10.2]{MR1711328}. In addition, $\mathcal{I}\le_{\mathrm{RK}} \mathcal{J}$ for all Erd{\H o}s--Ulam ideals $\mathcal{I}, \mathcal{J}$ (where an ideal $\mathcal{H}$ is said to be Erd{\H o}s--Ulam if there exists a nonnegative real sequence $(h_n)$ such that $\sum_n h_n=\infty$, $h_n=o(\sum_{k\le n}h_k)$ as $n\to \infty$, and $S \in \mathcal{H}$ if and only if $\sum_{k \in S, k\le n}h_k=o(\sum_{k\le n}h_k)$ as $n\to \infty$), see \cite[Lemma 1.13.10]{MR1711328}, cf. also \cite[Corollary 1]{MR3771234}. Other examples can be found within the class of summable ideals, see \cite[Section 1.12]{MR1711328}. 

Our last result extends the latter \cite[Theorem 2.4]{MR2241135} by finding all ideals $\mathcal{I}$ for which equality \eqref{eq:mainproblemequalitycores} holds with $\mathcal{J}=\mathrm{Fin}$ and some matrix $A$:
\begin{thm}\label{thm:JFin}
There exists an infinite real matrix $A \in (\ell_\infty, \ell_\infty)$ satisfying equality \eqref{eq:mainproblemequalitycores} with $\mathcal{J}=\mathrm{Fin}$ if and only if $\mathcal{I}=\mathrm{Fin}$ or $\mathcal{I}$ is an isomorphic copy of $\mathrm{Fin}\oplus \mathcal{P}(\omega)$ on $\omega$. 
\end{thm}

(Equivalently, the latter condition means $\mathcal{I}=\{S\subseteq \omega: S\cap T \in \mathrm{Fin}\}$ for some infinite $T\subseteq \omega$.)  
It is worth noting that Theorem \ref{thm:JFin}  
is also related to the question posed by Mazur in \emph{The Scottish Book} whether the notion of statistical convergence (i.e., $\mathcal{Z}$-convergence) of bounded sequences is equivalent to some matrix summability method, see \cite{MR3863065} and references therein. A positive answer has been given by Khan and Orhan in \cite[Theorem 2.2]{MR2340330}.

Based on the previous observations, we leave as an open question to check whether there exists a matrix $A \in (\ell_\infty, \ell_\infty)$ satisfying equality \eqref{eq:mainproblemequalitycores} with $\mathcal{I}=\mathrm{Fin}$ if and only if $\mathcal{J}$ is meager. In the same direction, it would be interesting to know if the condition $\mathcal{I} \le_{\mathrm{RK}} \mathcal{J}$ is also necessary in the statament of Theorem \ref{thm:RK}. 

\section{Proofs}

We start with an auxiliary lemma and, then, we proceed to the proofs of our results.
\begin{lem}\label{lem:equalitycoresc0}
Let $\bm{x}$ and $\bm{y}$ be two relatively compact sequences taking values in a locally convex topological vector space $X$. Let $\mathcal{I}$ be an ideal on $\omega$. 
Then 
$$
\mathrm{core}_{\bm{x}}(\mathcal{I})=\mathrm{core}_{\bm{y}}(\mathcal{I})
$$
whenever $\mathcal{I}\text{-}\lim (\bm{x}-\bm{y})=0$. 
\end{lem}
\begin{proof}
Since $X$ is, in particular, a topological group, it follows by \cite[Lemma 3.5]{MR3920799} and the hypothesis $\mathcal{I}\text{-}\lim (\bm{x}-\bm{y})=0$ that $\Gamma_{\bm{x}}(\mathcal{I})=\Gamma_{\bm{y}}(\mathcal{I})$. 
Now, let $K_{\bm{x}}$ be the closure of the image $\{x_n: n\in \omega\}$, which is a compact subset of the locally convex space $X$. Of course, this implies that $\{x_n: n \in \omega\}\subseteq K_{\bm{x}}+U$ for every open neighborhood $U$ of $0$ (which means that $\bm{x}$ is ``$\mathcal{I}^\star$-asymptotically $K_{\bm{x}}$-controlled,'' using the terminology in \cite[Definition 3.2]{MR4126774}). It follows by \cite[Theorem 3.4]{MR4126774} that $\mathrm{core}_{\bm{x}}(\mathcal{I})=\overline{\mathrm{co}}\,\Gamma_{\bm{x}}(\mathcal{I})$; cf. also \cite[Theorem 2.2]{MR3955010} in the case of first countable locally convex spaces. With an analog reasoning on the sequence $\bm{y}$, we conclude that $\mathrm{core}_{\bm{x}}(\mathcal{I})=\overline{\mathrm{co}}\,\Gamma_{\bm{x}}(\mathcal{I})=\overline{\mathrm{co}}\,\Gamma_{\bm{y}}(\mathcal{I})=\mathrm{core}_{\bm{y}}(\mathcal{I})$. 
\end{proof}

\medskip

\begin{proof}
[Proof of Theorem \ref{thm:countably generated}]
\textsc{Only If part.} Pick a sequence $\bm{x} \in c_b(\mathcal{I})$ and define $\eta:=\mathcal{I}\text{-}\lim \bm{x}$, 
so that $\Gamma_{\bm{x}}(\mathcal{I})=\{\eta\}$. 
It follows that $A\bm{x}$ is well-defined bounded sequence and, thanks to equality \eqref{eq:mainproblemequalitycores}, that 
$\mathrm{core}_{A\bm{x}}(\mathcal{J})=\mathrm{core}_{\bm{x}}(\mathcal{I})=\{\eta\}$. We conclude by \cite[Proposition 4.2]{MR3955010} that $\mathcal{J}\text{-}\lim A\bm{x}=\eta$, therefore $A$ is $(\mathcal{I}, \mathcal{J})$-regular. 

At this point, pick a set $E \in \mathcal{I}^+$. Recall that the hypothesis $A \in (\ell_\infty, \ell_\infty)$ is equivalent to $\sup_{n \in \omega} \sum_k|a_{n,k}|<\infty$, see e.g. \cite[Theorem 2.3.5]{MR1817226}. In addition, it follows by Theorem \ref{thm:mainsilvermantoeplitz} that $\mathcal{J}\text{-}\lim_n |a_{n,k}|=0$ for all $k \in \omega$. This implies that conditions (K1)--(K3) in the statement of \cite[Corollary 4.3]{Leo22} are satisfied (in the case where $d=m=1$, so that $a_{n,k}(i,j)$ is simply $a_{n,k}$ considering that both $i$ and $j$ can take only one value). Hence, thanks to  
\cite[Corollary 4.3]{Leo22}, there exists a $\{-1,0,1\}$-valued sequence $\bm{x}$ supported on $E$ such that 
\begin{equation}\label{eq:equalitylimsup}
\mathcal{J}\text{-}\limsup_{n\to \infty} \sum_{k \in E}|a_{n,k}|=\mathcal{J}\text{-}\limsup_{n\to \infty} \left|A_n\bm{x}\right|.
\end{equation}
(Again, equation \eqref{eq:equalitylimsup} is a rewriting of the claim in \cite[Corollary 4.3]{Leo22} taking into account that both variables $i$ and $j$ can take only value.) 
Define $F:=\{n \in E: x_n=1\}$ and $G:=\{n \in E: x_n=-1\}$, so that $\{F,G\}$ is a partition of $E$ and $\bm{x}:=\bm{1}_F -\bm{1}_G$ (hereafter, $\bm{1}_S$ stands for the characteristic function of $S$). Note that, since $\mathcal{I}$ is an ideal, then at least one between $F$ and $G$ is an $\mathcal{I}$-positive set. For convenience, let $h: \mathbf{R}\to \mathbf{R}$ be the function defined by $\eta \mapsto |\eta|$. Since $h$ is continuous, it follows by \cite[Proposition 3.2]{MR4428911} that 
$$
\Gamma_{h(A\bm{x})}(\mathcal{J})=h\left[\Gamma_{A\bm{x}}(\mathcal{J})\right],
$$
where $h(A\bm{x}):=(h(A_n\bm{x}): n \in \omega)$. 
Since $\bm{x}$ is a bounded real sequence, $A \in (\ell_\infty, \ell_\infty)$, and $h$ is continuous, then $h(A\bm{x})$ is a relatively compact sequence, so that the above sets are nonempty, see e.g. \cite[Lemma 3.1(vi)]{MR3920799}. In particular, it follows that $\max\, \Gamma_{h(A\bm{x})}(\mathcal{J}) 
= \max \, h\left[\Gamma_{A\bm{x}}(\mathcal{J})\right]$. 
Taking into account equality \eqref{eq:mainproblemequalitycores} (so that $\mathcal{J}\text{-}\limsup$ and $\mathcal{J}\text{-}\liminf$ are preserved) and that the sequence $\bm{x}=\bm{1}_F-\bm{1}_G$ has at least a $\mathcal{J}$-cluster point in $\{-1,1\}$, we conclude that 
\begin{displaymath}
\begin{split}
\mathcal{J}\text{-}\limsup_{n\to \infty} \sum_{k \in E}|a_{n,k}|&=\max\, \Gamma_{h(A\bm{x})}(\mathcal{J}) 
= \max \, h\left[\Gamma_{A\bm{x}}(\mathcal{J})\right]\\
&=\max\{|\mathcal{J}\text{-}\limsup A\bm{x}|, |\mathcal{J}\text{-}\liminf A\bm{x}|\}\\
&=\max\{|\mathcal{J}\text{-}\limsup \bm{x}|, |\mathcal{J}\text{-}\liminf \bm{x}|\}
=\max \, h\left[\Gamma_{\bm{x}}(\mathcal{J})\right] =1.
\end{split}
\end{displaymath}
Therefore both conditions \ref{item:Leo1} and \ref{item:Leo2} hold.

\medskip

\textsc{If part.} Conversely, let $A =(a_{n,k}: n,k \in \omega) \in (\ell_\infty, \ell_\infty)$ be a $(\mathcal{I}, \mathcal{J})$-regular matrix which satisfies condition \ref{item:Leo2}. Then we get
$$
1=\mathcal{J}\text{-}\lim_{n\to \infty}\sum_{k \in \omega}a_{n,k} 
\le \mathcal{J}\text{-}\liminf_{n\to \infty} \sum_{k\in \omega} |a_{n,k}|
\le \mathcal{J}\text{-}\limsup_{n\to \infty} \sum_{k\in \omega} |a_{n,k}|=1,
$$
so that $\mathcal{J}\text{-}\lim_n \sum_k |a_{n,k}|=\mathcal{J}\text{-}\lim_n \sum_k a_{n,k}=1$. Decomposing each $a_{n,k}$ into its positive and negative part as $a_{n,k}^+-a_{n,k}^-$ for all $n,k \in \omega$, it follows that 
\begin{equation}\label{eq:positivenegativepart}
\mathcal{J}\text{-}\lim_{n\to \infty} \sum_{k\in \omega} a_{n,k}^-=0
\quad \text{ and }\quad 
\mathcal{J}\text{-}\lim_{n\to \infty} \sum_{k\in \omega} a_{n,k}^+=1.
\end{equation}

At this point, pick a sequence $\bm{x} \in \ell_\infty$, define $\eta:=\mathcal{I}\text{-}\limsup \bm{x}$. Considering that $\mathrm{core}_{\bm{x}+\kappa \bm{1}_\omega}(\mathcal{I})=\mathrm{core}_{\bm{x}}(\mathcal{I})+\{\kappa\}$ and also, by the $(\mathcal{I},\mathcal{J})$-regularity of $A$, that $\mathrm{core}_{A(\bm{x}+\kappa \bm{1}_\omega)}(\mathcal{J})=\mathrm{core}_{A\bm{x}}(\mathcal{J})+\{\kappa\}$ for all $\kappa \in \mathbf{R}$, we can suppose without loss of generality that $\eta>0$. Fix an arbitrary $\varepsilon>0$ and define 
$$
\delta:=\min\left\{\,\frac{\varepsilon}{2+\eta+4\|x\|},\,1\right\}
\quad \text{ and }\quad 
E:=\left\{k \in \omega: x_k\ge \eta-\delta\right\}.
$$
Note that $\delta>0$ and $E$ is an $\mathcal{I}$-positive set since $\eta$ is an $\mathcal{I}$-cluster point of $\bm{x}$. It follows by \eqref{eq:positivenegativepart} and condition \ref{item:Leo2} that $\mathcal{J}\text{-}\limsup_n \sum_{k \in E}a_{n,k}^+=1$. Thus, define 
\begin{equation}\label{eq:definitionS}
S:=\left\{n \in \omega: 
1-\delta \le \sum_{k \in E}a_{n,k}^+ \le \sum_{k \in \omega}|a_{n,k}| \le 1+\delta
\right\}.
\end{equation}
Observe the first inequality in the definition of $S$ holds on a $\mathcal{J}$-positive set, the second one for all $n$, and the latter one on $\mathcal{J}^\star$. Therefore $S \in \mathcal{J}^+$. For each $n \in S$, it also follows that $|\sum_{k \in \omega}a_{n,k}^+-1|\le \delta$, so that 
\begin{equation}\label{eq:estimatesupperbounds}
\left|\,\sum_{k \in \omega}a_{n,k}^-\,\right| \le 2\delta 
\quad \text{ and }\quad 
\left|\,\sum_{k \in E^c}a_{n,k}^+\,\right| \le 2\delta. 
\end{equation}
Putting all together, we obtain that, for all $n \in S$, 
\begin{displaymath}
\begin{split}
A_n\bm{x}&=\sum_{k \in E}a_{n,k}^+x_k+\sum_{k \in E^c}a_{n,k}^+x_k-\sum_{k \in \omega}a_{n,k}^-x_k \\
&\ge (1-\delta)(\eta-\delta)-2\delta \|x\| -2\delta \|x\|\\
&\ge \eta -\delta (1+\eta+4\|x\|) \ge \eta-\varepsilon.
\end{split}
\end{displaymath}

At the same time, define $E^\prime:=\{k \in \omega: x_k \le \eta+\delta\}$, which belongs to $\mathcal{I}^\star$ and note, similarly as above, that $\mathcal{J}\text{-}\lim_n \sum_{k \in E^\prime} |a_{n,k}|
=\mathcal{J}\text{-}\lim_n \sum_{k \in E^\prime} a_{n,k}=\mathcal{J}\text{-}\lim_n \sum_{k \in E^\prime} a_{n,k}^+=1$ and $\mathcal{J}\text{-}\lim_n \sum_{k \in \omega} a_{n,k}^-=0$. Let $S^\prime$ be the set defined as in \eqref{eq:definitionS} replacing $E$ with $E^\prime$, and note that $S^\prime \in \mathcal{J}^\star$. Similarly, estimates \eqref{eq:estimatesupperbounds} hold for all $n \in S^\prime$ replacing $E$ with $E^\prime$. Putting again all together, we obtain that, for all $n \in S^\prime$, 
\begin{displaymath}
\begin{split}
A_n\bm{x}&=\sum_{k \in E}a_{n,k}^+x_k+\sum_{k \in E^c}a_{n,k}^+x_k-\sum_{k \in \omega}a_{n,k}^-x_k \\
&\le (1+\delta)(\eta+\delta)+2\delta \|x\| +2\delta \|x\|\\
&\le \eta +\delta (2+\eta+4\|x\|) \le \eta+\varepsilon.
\end{split}
\end{displaymath}

Since $\varepsilon$ is arbitrary, we conclude that $\mathcal{J}\text{-}\limsup A\bm{x}= \eta$. Therefore $A$ preserves the ideal superior limits for all bounded sequences $\bm{x}$. Replacing $\bm{x}$ with $-\bm{x}$, $A$ preserves also the corresponding ideal inferior limits. It follows by identity \eqref{eq:Icorerepresentation} that equality \eqref{eq:mainproblemequalitycores} holds, concluding the proof. 
\end{proof}

\begin{rmk}\label{rmk:generalizationifonlyif}
It is clear from the proof above that the \textsc{If part} holds without any additional hypothesis on $\mathcal{J}$. Moreover, the fact the $\mathcal{J}$ is countably generated has been used only once in the proof of the \textsc{Only If part}, precisely in the existence of a $\{-1,0,1\}$-valued sequence $\bm{x}$ supported on a given $E \in \mathcal{I}^+$ and satisfying equality \eqref{eq:equalitylimsup}. The latter is trivial if $A\ge 0$ by choosing $\bm{x}=\bm{1}_E$. In this sense, we recover also Theorem \ref{thm:connorfridyorhan}. 
\end{rmk}

\begin{rmk}\label{rmk:alwaysholds}
On the other hand, if 
$\mathcal{I}=\mathcal{J}=\mathcal{Z}$, 
then the analogue of Theorem \ref{thm:countably generated} does \emph{not} hold. Indeed, thanks to the proof of \cite[Theorem 1.4]{ConnorLeo} there exists a matrix $A \in (\ell_\infty, c_0(\mathcal{Z}) \cap \ell_\infty)$ 
and an infinite set $I \in \mathcal{I}$ such that 
$$
\mathcal{Z}\text{-}\limsup_{n\to \infty} \sum_{k \in I}|a_{n,k}|=1. 
$$ 
(Here, $c_0(\mathcal{Z})$ stands for the vector space of sequences which are $\mathcal{Z}$-convergent to $0$.) 
At this point, define $B:=A+\mathrm{Id}$, where $\mathrm{Id}$ stands for the infinite identity matrix. On one hand, for each $\bm{x} \in \ell_\infty$ we have $B\bm{x}=A\bm{x}+\bm{x}$ and $A\bm{x} \in c_0(\mathcal{Z}) \cap \ell_\infty$, hence by Lemma \ref{lem:equalitycoresc0} we get $\mathrm{core}_{B\bm{x}}(\mathcal{Z})=\mathrm{core}_{\bm{x}}(\mathcal{Z})$. Thus equality \eqref{eq:mainproblemequalitycores} holds for the matrix $B$. On the other hand,  
$$
\mathcal{Z}\text{-}\limsup_{n\to \infty} \sum_{k \in \omega}|b_{n,k}|=1+\mathcal{Z}\text{-}\limsup_{n\to \infty} \sum_{k \in \omega}|a_{n,k}| \ge 
1+\mathcal{Z}\text{-}\limsup_{n\to \infty} \sum_{k \in I}|a_{n,k}|=2.
$$
This shows that $B$ does not satisfy condition \ref{item:Leo2}. 
\end{rmk}

\medskip

\begin{proof}
[Proof of Theorem \ref{thm:RK}] 
Since $\mathcal{I} \le_{\mathrm{RK}} \mathcal{J}$, there exists a map $h: \omega \to \omega$ such that $S \in \mathcal{I}$ if and only if $h^{-1}[S] \in \mathcal{J}$. Now, let $A=(a_{n,k}: n,k \in \omega)$ be the matrix defined by
\begin{displaymath}
a_{n,k}=
\begin{cases}
\,1\,\,\,\,& \text{ if }k=h(n),\\
\,0 & \text{ otherwise}.
\end{cases}
\end{displaymath}
Note that $A \in (\ell_\infty, \ell_\infty)$ since every row contains a single $1$ (however, $A$ is not necessarily regular if the witnessing map $h$ cannot be chosen finite-to-one). 
Fix $\bm{x} \in \ell_\infty$, let $U\subseteq \mathbf{R}$ be a nonempty open set, and define $S:=\{n \in \omega: x_n \in U\}$. Observe also that $A_n\bm{x}=\sum_ka_{n,k}x_k=x_{h(n)}$ for all $n \in \omega$. It follows that $S \in \mathcal{I}$ if and only if 
$$
h^{-1}[S]=\{n \in \omega: x_{h(n)} \in U\}=\{n \in \omega: A_n\bm{x} \in U\} \in \mathcal{J}.
$$
This implies that $\Gamma_{\bm{x}}(\mathcal{I})=\Gamma_{A\bm{x}}(\mathcal{J})$, so that by \cite[Corollary 2.3]{MR3955010} we get
$$
\mathrm{core}_{\bm{x}}(\mathcal{I})
=\mathrm{co}(\Gamma_{\bm{x}}(\mathcal{I}))
=\mathrm{co}(\Gamma_{A\bm{x}}(\mathcal{J}))
=\mathrm{core}_{A\bm{x}}(\mathcal{J}).
$$
Therefore equality \eqref{eq:mainproblemequalitycores} holds. 
\end{proof}

\medskip

For our last proof, we need to recall that an ideal $\mathcal{I}$ on $\omega$ is said to be:
\begin{enumerate}[label={\rm (\roman{*})}]
\item a $P$\emph{-ideal} if for all increasing sequences $(I_n)$ in $\mathcal{I}$ there exists $I \in \mathcal{I}$ such that $I_n\setminus I \in \mathrm{Fin}$ for all $n \in \omega$; 
\item a $P^+$\emph{-ideal} if for all decreasing sequences $(I_n)$ in $\mathcal{I}^+$ there exists $I \in \mathcal{I}^+$ such that $I\setminus I_n \in \mathrm{Fin}$ for all $n \in \omega$; 
\item \emph{tall} if every infinite set $S\subseteq \omega$ contains an infinite subset $I\subseteq S$ such that $I \in \mathcal{I}$;
\item \emph{nowhere tall} if, for every $S \in \mathcal{I}^+$, the ideal $\mathcal{I}\upharpoonright S:=\mathcal{I} \cap \mathcal{P}(S)$ is not tall.
\end{enumerate}
\begin{proof}
[Proof of Theorem \ref{thm:JFin}]
\textsc{If part.} 
Suppose that $\mathcal{I}=\{S\subseteq \omega: S \cap T \in \mathrm{Fin}\}$ for some infinite $T\subseteq \omega$ (i.e., $\mathcal{I}=\mathrm{Fin}$ if $T$ is cofinite or $\mathcal{I}$ is an isomorphic copy of $\mathrm{Fin}\oplus \mathcal{P}(\omega)$ on $\omega$ otherwise). Then $\mathcal{I} \le_{\mathrm{RK}} \mathrm{Fin}$. In fact, if $(t_n: n \in \omega)$ denotes the increasing enumeration of elements of $T$, one can choose the witnessing map $h: \omega\to \omega$ defined by $h(n)=t_n$ for all $n \in \omega$. 
The claim follows by Theorem \ref{thm:RK}.

\medskip

\textsc{Only If part.} Let $\mathcal{I}$ be an ideal on $\omega$ for which there exists a matrix $A \in (\ell_\infty, \ell_\infty)$ which satisfies equality \eqref{eq:mainproblemequalitycores} with $\mathcal{J}=\mathrm{Fin}$. We divide the remaining proof in several claims.  

\begin{claim}\label{claim1:Gdeltasigmadelta}
$\mathcal{I}$ is 
an analytic ideal. 
\end{claim}
\begin{proof}
Observe that, for each $E\subseteq \omega$, we have $E \in \mathcal{I}$ if and only if $\mathcal{I}\text{-}\limsup \bm{1}_E\le 0$. It follows by equality \eqref{eq:mainproblemequalitycores} that 
$$
\mathcal{I}=\left\{E\subseteq \omega: \limsup A\bm{1}_E\le 0\right\}=\bigcap_{p \in \omega} \bigcup_{q \in \omega} \bigcap_{n \ge q} G_{p,n},
$$
where $G_{p,n}:=\left\{E\subseteq \omega: \sum_{k \in E}a_{n,k}<2^{-p}\right\}$ for all $n,p \in \omega$. Hence, it is sufficient to show that each $G_{n,p}$ is open. For, fix $n,p \in \omega$. If $G_{n,p}=\emptyset$, then it is open. Otherwise fix $E \in G_{n,p}$. Since $\sum_{k}|a_{n,k}|\le \|A\|<\infty$, there exists $k_0 \in \omega$ such that 
$$
\sum_{k>k_0}|a_{n,k}|<\frac{1}{2}\left(2^{-p}-\sum_{k \in E}a_{n,k}\right).
$$
Now, let $F\subseteq \omega$ be a set such that $E\cap [0,k_0]=F\cap [0,k_0]$. It follows that 
\begin{displaymath}
\begin{split}
\sum_{k \in F}a_{n,k}
&\le \sum_{k \in F\cap [0, k_0]}a_{n,k}+\sum_{k \in F\setminus [0, k_0]}|a_{n,k}|\\
&\le \sum_{k \in E\cap [0, k_0]}a_{n,k}+\sum_{k >k_0}|a_{n,k}|\\
&\le \sum_{k \in E}a_{n,k} +2 \sum_{k >k_0}|a_{n,k}|<2^{-p}.\\
\end{split}
\end{displaymath}
This shows that $F \in G_{n,p}$. Hence $E$ is an interior point, so that $G_{n,p}$ is open. Therefore $\mathcal{I}$ is a $G_{\delta \sigma \delta}$-ideal.
\end{proof}

\medskip

\begin{claim}\label{claim:P+ideal}
$\mathcal{I}$ is a $P^+$-ideal. 
\end{claim}
\begin{proof}
Let us suppose for the sake of contradiction that $\mathcal{I}$ is not a $P^+$-ideal, hence it is possible to fix a strictly decreasing sequence $(I_n: n \in \omega)$ in $\mathcal{I}^+$ such that for all sequences $(F_n: n \in \omega)$ of finite sets with $F_n \subseteq I_{n}\setminus I_{n+1}$ we have $\bigcup_n F_n \in \mathcal{I}$. Since $I_0 \in \mathcal{I}^+$ it follows by equality \eqref{eq:mainproblemequalitycores} that $\limsup A\bm{1}_{I_0}=1$. Hence there exists $n_0 \in \omega$ such that $A_{n_0}\bm{1}_{I_0}>1-2^{-0}$. Set $p_0:=0$ and pick an integer $q_0>p_0$ such that $\sum_{k \in F_0}a_{n_0,k}\bm{1}_{I_0}(k)>1-2^{-0}$, where $F_0:=\omega \cap [p_0,q_0]$. Recall also that $A$ is $(\mathcal{I}, \mathrm{Fin})$-regular since it satisfies \eqref{eq:mainproblemequalitycores}, hence it is regular. In particular, by Theorem \ref{thm:mainsilvermantoeplitz}, we have $\lim_k a_{n,k}=0$ and $\sum_k|a_{n,k}|<\infty$ for all $n,k \in \omega$. 
At this point, suppose that $n_{i-1}$ and $F_{i-1}:=\omega \cap [p_{i-1},q_{i-1}]$ have been defined for some positive integer $i$. Then, proceed recursively as follows: 
\begin{enumerate}[label={\rm (\roman{*})}]
\item Pick an integer $p_{i}>q_{i-1}$ with the property that 
$$
\sum_{k\ge p_{i}}|a_{n_j,k}|<2^{-i} 
$$
for all $j \in \omega \cap [0,i-1]$. 
\item Let $n_{i}>n_{i-1}$ be an integer such that 
$$
A_{n_i}\bm{1}_{I_i}>1-2^{-i} 
\quad \text{ and }\quad 
\sum_{k< p_{i}}|a_{n,k}|<2^{-i}
$$
for all integers $n\ge n_i$.
\item Let $q_i>p_i$ be an integer such that 
$$
\sum_{k \in F_i}a_{n_i,k}\bm{1}_{I_i}(k)>1-2^{1-i},
$$
where $F_i:=\omega \cap [p_i,q_i]$. (Note that this is possible because $\sum_{k\ge p_i}a_{n_i,k}\bm{1}_{I_i}(k)$ is at least 
$A_{n_i}\bm{1}_{I_i} - \sum_{k<p_i}|a_{n_i,k}|>1-2^{1-i}$.)
\end{enumerate}
To conclude, define $F:=\bigcup_i F_i \cap I_i$. By the standing hypothesis, we have $F \in \mathcal{I}$, hence by equality \eqref{eq:mainproblemequalitycores} we get $\limsup A\bm{1}_F=\mathcal{I}\text{-}\limsup \bm{1}_F=0$. On the other hand, it follows by the construction above that, for all $i\ge 1$, 
\begin{displaymath}
\begin{split}
\limsup_{n\to \infty} A_n\bm{1}_F 
&\ge \limsup_{i\to \infty}\sum_{k \in F}a_{n_i,k} \\
&\ge \limsup_{i \to \infty}\left(\sum_{k \in F_i \cap I_i}a_{n_i,k}-\sum_{k< p_i}|a_{n_i,k}|- \sum_{k\ge p_{i+1}}|a_{n_i,k}|\right)\\
&\ge \limsup_{i\to \infty}(1-2^{1-i}-2^{-i}-2^{-1-i})=1.
\end{split}
\end{displaymath}
This contradiction proves that $\mathcal{I}$ is a $P^+$-ideal. 
\end{proof}

\medskip

\begin{claim}\label{claim:Pideal}
$\mathcal{I}$ is a $P$-ideal. 
\end{claim}
\begin{proof}
Let us suppose for the sake of contradiction that $\mathcal{I}$ is not a $P$-ideal, hence it is possible to fix an increasing sequence $(I_n: n \in \omega)$ in $\mathcal{I}$ such that, for all sequences $(F_n: n \in \omega)$ with $F_n \subseteq D_n:=I_{n+1}\setminus I_n$ for each $n$, we have $\bigcup_n (D_n \setminus F_n) \in \mathcal{I}^+$. Without loss of generality, we can assume that $I_0=\emptyset$. Define 
$$
S:=\{n \in \omega: D_n \notin \mathrm{Fin}\}.
$$ 
It is easy to see that, if $S$ is finite, then $(I_n: n \in \omega)$ cannot be a sequence witnessing that $\mathcal{I}$ is not a $P$-ideal: indeed, in such case, $I:=I_0$ if $S=\emptyset$ or $I:=I_{1+\max S}$ if $S\neq \emptyset$ satisfies $I_n\setminus I \in \mathrm{Fin}$ for all $n \in \omega$. Hence $S$ has to be infinite, which implies that, passing to a suitable subsequence, we can assume without loss of generality that $D_n$ is infinite for all $n \in \omega$. 

Now, note that, since each $I_n$ belongs to $\mathcal{I}$, then $\lim A\bm{1}_{I_n}=0$ for all $n \in \omega$ by equality \eqref{eq:mainproblemequalitycores}. 
Let $(k_n: n \in \omega)$ be a strictly increasing sequence in $\omega$ such that 
\begin{equation}\label{eq:Pideal1}
\sum_{k>k_n}|a_{n,k}|<2^{-n} \quad \text{ for all }n \in \omega. 
\end{equation}
It follows by the \textsc{If part} in the proof of Theorem \ref{thm:countably generated} with $\mathcal{J}=\mathrm{Fin}$ that $\lim_n \sum_k a_{n,k}^-=0$. In particular, there exists a strictly increasing sequence $(h_m: m \in \omega)$ such that 
\begin{equation}\label{eq:Pideal2}
\sum_{k \in \omega}a_{n,k}^- <2^{-m} \quad 
\text{ for all }n\ge h_m. 
\end{equation}
Let also $(t_m: m \in \omega)$ be a strictly increasing sequence in $\omega$ such that 
\begin{equation}\label{eq:Pideal3}
t_m\ge h_m \text{ and }
\left|\,\sum_{i\le m}A_n\bm{1}_{D_i}\right|
<2^{-m} 
\quad \text{ for all }m \in \omega \text{ and }n\ge t_m.
\end{equation}
To conclude, define $F_n:=\omega \cap [0,k_{t_{n}}]$ for all $n \in \omega$ and set $D_\infty:=\bigcup_n (D_n\setminus F_n)$. On the one hand, it follows by the standing hypothesis that $D_\infty \in \mathcal{I}^+$, hence by equality \eqref{eq:mainproblemequalitycores} we have $\limsup A\bm{1}_{D_\infty}=1$. On the other hand, pick $m \in \omega$ and fix $n\in [t_m, t_{m+1})$. It follows 
 that 
\begin{displaymath}
\begin{split}
 \left| A_n \bm{1}_{D_\infty} \right|
 &= \left|\, \sum_{i \in \omega}  A_n \bm{1}_{D_i\setminus F_i}\right| \\
& \le \sum_{k>k_n}|a_{n,k}|+\left|\, \sum_{i \in \omega}\sum_{k\le k_n}a_{n,k}\bm{1}_{D_i\setminus F_i}(k)\,\right| \\
& \le 2^{-n}+ \left|\, \sum_{k\le k_n}\sum_{i \in \omega}a_{n,k}\bm{1}_{D_i\setminus F_i}(k)\right|, \\
\end{split}
\end{displaymath}
where at the last inequality we used \eqref{eq:Pideal1}. 
At this point, notice that, if $k\le k_n$ and $i>m$ then $\bm{1}_{D_i\setminus F_i}(k)=0$ since $\min (D_i\setminus F_i) > \max (F_i) = k_{t_{i}}\ge k_{t_{m+1}}>k_n$. Taking into account that $n\ge t_m \ge m$, inequality \eqref{eq:Pideal2}, and that $t_m\ge h_m$, we obtain 
\begin{displaymath}
\begin{split}
 \left| A_n \bm{1}_{D_\infty} \right| 
& \le 2^{-m}+ \left|\, \sum_{k\le k_n}\sum_{i \le m}a_{n,k}\bm{1}_{D_i\setminus F_i}(k)\right|\\
& \le 2^{-m}+ \left|\, \sum_{k\le k_n}\sum_{i \le m}a_{n,k}\bm{1}_{D_i}(k)\right|+\sum_{k \in \omega} a_{n,k}^-\\
& \le 2^{1-m}+ \left|\, \sum_{k\le k_n}\sum_{i \le m}a_{n,k}\bm{1}_{D_i}(k)\right|.
\end{split}
\end{displaymath}
Lastly, using also inequality \eqref{eq:Pideal3}, we get
\begin{displaymath}
\begin{split}
 \left| A_n \bm{1}_{D_\infty} \right| 
& \le 2^{1-m}+ \left|\, \sum_{k\in \omega}\sum_{i \le m}a_{n,k}\bm{1}_{D_i}(k)\right| + \left|\sum_{k>k_n} \sum_{i\le m} a_{n,k} \bm{1}_{D_i}(k)\right|\\
& \le 2^{1-m}+ \left|\, \sum_{i \le m}A_n\bm{1}_{D_i}\right| + \sum_{k>k_n} |a_{n,k}|\\
& \le 2^{1-m}+2^{-m}+2^{-n} \le 4^{1-m}.
\end{split}
\end{displaymath}
This proves that $\lim A\bm{1}_{D_\infty}=0$, which gives the desired contradiction. 
\end{proof}

\medskip

Thanks to Claims \ref{claim1:Gdeltasigmadelta}, \ref{claim:P+ideal}, and \ref{claim:Pideal}, $\mathcal{I}$ is an analytic $P$-ideal which is also a $P^+$-ideal.
 Although it will not be used in the following results, it follows by \cite[Theorem 2.5]{MR3883171} that $\mathcal{I}$ is a necessarily $F_\sigma$ $P$-ideal (we omit details). 

\medskip

\begin{claim}\label{claim:Pnottall}
$\mathcal{I}$ is not tall. 
\end{claim}
\begin{proof}
Let us suppose for the sake of contradiction that $\mathcal{I}$ is tall. Define the infinite matrix $A^+:=(a_{n,k}^+: n,k \in \omega)$. Since $\lim_n \sum_k a_{n,k}^-=0$ (cf. the proof of Claim \ref{claim:Pideal}), it follows that $A^+$ is a nonnegative $(\mathcal{I}, \mathrm{Fin})$-regular matrix; in particular, it is a nonnegative regular matrix. At this point, define the map $\mu^\star: \mathcal{P}(\omega) \to \mathbf{R}$ by 
$$
\mu^\star(S):=\limsup A^+ \bm{1}_S 
\quad \text{ for all }S \subseteq \omega.
$$
Note also that $\lim (A\bm{x}-A^+\bm{x})=0$ for all $\bm{x} \in \ell_\infty$, hence by Lemma \ref{lem:equalitycoresc0} and equality \eqref{eq:mainproblemequalitycores} 
$$
\mathrm{core}_{A^+\bm{x}}(\mathrm{Fin})=\mathrm{core}_{\bm{x}}(\mathcal{I}) 
\quad 
\text{ for all sequences }\bm{x} \in \ell_\infty.
$$
Thus $\mathcal{I}=\{S\subseteq \omega: \mathcal{I}\text{-}\lim \bm{1}_S=0\}=\{S\subseteq \omega: \lim A^+ \bm{1}_S=0\}=\{S\subseteq \omega: \mu^\star(S)=0\}$. Since $\mathcal{I}$ is not tall, it follows by \cite[Proposition 7.2]{MR1865750} that $\lim_n \sup_k a_{n,k}^+=0$. In addition, recalling that $A^+$ is a nonnegative matrix, we have also that $\lim_n a_{n,k}^+=0$ and $\sum_k a_{n,k}^+<\infty$ for all $n,k \in \omega$ by Theorem \ref{thm:mainsilvermantoeplitz}. It follows by \cite[Theorem 6.2]{MR1865750} that the function $\mu^\star$ has the weak Darboux property, i.e., for each $S\subseteq \omega$ and $y \in [0,\mu^\star(S)]$ there exists $X\subseteq S$ such that $\mu^\star(X)=y$, cf. \cite[Section 2]{MR3597402}. This implies that there exists a decreasing sequence $(I_m: m \in \omega)$ of subsets of $\omega$ such that 
$$
\mu^\star(I_m)=2^{-m}
\quad \text{ for all }m \in \omega. 
$$
At this point, let $I\subseteq \omega$ be a set such that $J_m:=I\setminus I_m \in \mathrm{Fin}$ for all $m \in \omega$. Since $\mu^\star$ is monotone and subadditive, we obtain
$$
\mu^\star(I) \le \mu^\star(I_m)+\mu^\star(J_m)=2^{-m}+\limsup_{n\to \infty} \sum_{k \in J_m}a_{n,k}=2^{-m}
\quad \text{ for all }m \in \omega.
$$
Hence $\mu^\star(I)=0$, i.e., $I \in \mathcal{I}$. This proves that $\mathcal{I}$ is not a $P^+$-ideal, which contradicts Claim \ref{claim:P+ideal}. Therefore $\mathcal{I}$ cannot be tall.
\end{proof}

\medskip

\begin{claim}\label{claim:Pfrechet}
$\mathcal{I}$ is a nowhere tall ideal.  
\end{claim}
\begin{proof}
Fix a set $S \in \mathcal{I}^+$ and consider the ideal $\tilde{\mathcal{I}}:=\mathcal{I}\upharpoonright S$. Since $\mathcal{I}$ is analytic by Claim \ref{claim1:Gdeltasigmadelta} and $\mathcal{P}(S)$ is closed, then $\tilde{\mathcal{I}}$ is analytic as well. Moreover, since $\mathcal{I}$ is both a $P$-ideal and $P^+$-ideals by Claims \ref{claim:P+ideal} and \ref{claim:Pideal}, respectively, it is immediate that the same properties hold for $\tilde{\mathcal{I}}$. Let $\tilde{A}=(\tilde{a}_{n,k}: n,k \in \omega)$ be the matrix defined by $\tilde{a}_{n,k}:=a_{n,k}$ if $k \in S$ and $\tilde{a}_{n,k}:=0$ otherwise. Now, note that, by equality \eqref{eq:mainproblemequalitycores}, 
\begin{displaymath}
\begin{split}
\tilde{\mathcal{I}}&=\{X\subseteq S: \mathcal{I}\text{-}\lim \bm{1}_X=0\}=\{X\subseteq S: \lim A\bm{1}_X=0\}\\
&=\{X\subseteq S: \lim \tilde{A}\bm{1}_X=0\}=\{X\subseteq S: \tilde{\mu}^\star(X) =0\},
\end{split}
\end{displaymath}
where $\tilde{\mu}^\star(X):=\limsup_n \tilde{A}^+ \bm{1}_X$ for each $X\subseteq S$. Lastly, observe that $\tilde{A}^+$ has the following properties: $\limsup_n \sum_k \tilde{a}_{n,k}^+=1$, $\lim_n \tilde{a}_{n,k}^+=0$, and $\sum_k \tilde{a}_{n,k}^+<\infty$ for all $n,k \in \omega$. In particular, even if $\tilde{A}^+$ is not necessarily regular, it satisfies the hypotheses of \cite[Theorem 6.2 and Proposition 7.2]{MR1865750}. Hence, we proceed verbatim as in Claim \ref{claim:Pnottall} and we obtain that $\tilde{\mathcal{I}}$ is not tall. Since $S$ is arbitrary, we conclude that $\mathcal{I}$ is a nowhere tall ideal. 
\end{proof}

\medskip

Thanks to Claims \ref{claim1:Gdeltasigmadelta}, \ref{claim:Pideal}, and \ref{claim:Pfrechet} 
we know that $\mathcal{I}$ is an analytic $P$-ideal which is also  nowhere tall (notice that such properties are invariant under isomorphisms). 
Then, it is known that $\mathcal{I}$ is necessarily $\mathrm{Fin}$ or (an isomorphic copy on $\omega$ of) $\mathrm{Fin}\oplus \mathcal{P}(\omega)$ or $\{\emptyset\} \times \mathrm{Fin}$, see e.g. \cite[Corollary 1.2.11]{MR1711328} or \cite[Theorem 2.26]{MR4041540}. Finally, $\mathcal{I}$ has to be also a $P^+$-ideal by Claim \ref{claim:P+ideal}. Hence, it is immediate to check that $\mathcal{I}$ cannot be a copy of $\{\emptyset\} \times \mathrm{Fin}$. This concludes the proof.
\end{proof}

\begin{rmk}
Pick an infinite set $T\subseteq \omega$ which is not cofinite  
and define $\mathcal{J}:=\{S\subseteq \omega: S\cap T \in \mathrm{Fin}\}$ (hence, $\mathcal{J}$ is an isomorphic copy on $\omega$ of $\mathrm{Fin}\oplus \mathcal{P}(\omega)$). Pick also an ideal $\mathcal{I}$ on $\omega$ such that $\mathcal{I}\le_{\mathrm{RK}}\mathcal{J}$. 
Thanks to Theorem \ref{thm:RK}, 
there exists an infinite matrix $A \in (\ell_\infty, \ell_\infty)$ such that equality \eqref{eq:mainproblemequalitycores} holds. 
Now, as it has been observed in the proof of the \textsc{If part} of Theorem \ref{thm:JFin}, we have also $\mathcal{J} \le_{\mathrm{RK}} \mathrm{Fin}$. Hence, with the same argument,  there exists a matrix $B \in (\ell_\infty, \ell_\infty)$ such that $\mathrm{core}_{B\bm{x}}(\mathrm{Fin})=\mathrm{core}_{\bm{x}}(\mathcal{J})$ for all sequences $\bm{x} \in \ell_\infty$. 
In addition, by the proof of Theorem \ref{thm:RK}, it is possible to assume that each row of $B$ contains a single $1$. 
Set $C:=BA$ (observe that each entry of $C$ is well defined) and note that, if $\bm{x} \in \ell_\infty$ then $C\bm{x}=B(A\bm{x})$ is bounded, hence $C \in (\ell_\infty, \ell_\infty)$. It follows also that 
$$
\forall \bm{x} \in \ell_\infty, \quad 
\mathrm{core}_{C\bm{x}}(\mathrm{Fin})
=\mathrm{core}_{A\bm{x}}(\mathcal{J})
=\mathrm{core}_{\bm{x}}(\mathcal{I}). 
$$
This does not contradict the claim of Theorem \ref{thm:JFin}. In fact, we claim that, if $\mathcal{I}\le_{\mathrm{RK}}\mathcal{J}$ then either $\mathcal{I}=\mathrm{Fin}$ or $\mathcal{I}$ is an isomorphic copy of $\mathrm{Fin}\oplus \mathcal{P}(\omega)$.

For, let $h: \omega \to \omega$ be a map which witnesses $\mathcal{I}\le_{\mathrm{RK}}\mathcal{J}$. 
Observe that $W:=h[T] \in \mathcal{I}^+$ since $T \in \mathcal{J}^+$, hence $W$ is an infinite set. Considering that $h^{-1}[F] \in \mathrm{Fin}$ if and only if $F \in \mathrm{Fin}$ for each $F\subseteq \omega$, we obtain that 
\begin{displaymath}
\begin{split}
\forall S \in \omega, \quad 
S \in \mathcal{I} 
&\,\,\Longleftrightarrow\,\, h^{-1}[S] \in \mathcal{J}
\,\,\Longleftrightarrow\,\, h^{-1}[S] \cap T \in \mathrm{Fin}\\
&\,\,\Longleftrightarrow\,\, h^{-1}[S \cap W] \in \mathrm{Fin}
\,\,\Longleftrightarrow\,\, S \cap W \in \mathrm{Fin}.
\end{split}
\end{displaymath}
To sum up, $W$ is an infinite set and $\mathcal{I}=\{S\subseteq \omega: S \cap W \in \mathrm{Fin}\}$. Therefore either $\mathcal{I}=\mathrm{Fin}$ (in the case where $W$ is cofinite) or $\mathcal{I}$ is isomorphic to $\mathrm{Fin}\oplus \mathcal{P}(\omega)$ (in the opposite case). 
\end{rmk}

\section{Acknowledgments} 

The author is grateful to an anonymous referee for a careful reading of the manuscript and several useful observations.

\bibliographystyle{amsplain}

\end{document}